\documentclass[12pt,reqno]{amsart}
\usepackage{latexsym,amsmath}
\usepackage{amsthm}
\usepackage{amssymb}
\usepackage{amsthm}
\newtheorem{thm}{Theorem}

\newtheorem{lemma}[thm]{Lemma}
\newtheorem{cor}[thm]{Corollary}

\newcommand{\al}{\alpha}
\newcommand{\be}{\beta}
\newcommand{\del}{\delta}

\newcommand{\eps}{\epsilon}
\newcommand{\seps}{\sqrt{\epsilon}}
\newcommand{\hV}{\hat V}
\newcommand{\hv}{\hat v}

\newcommand{\tV}{\tilde V}

\newcommand{\ran}{\rangle\hspace{-2pt}\rangle}
\newcommand{\lan}{\langle\hspace{-2pt}\langle}
\newcommand{\od}[1]{\langle#1\rangle}
\newcommand{\ev}[1]{\lan#1\ran}
\newcommand{\Ca}{{\mathcal C}}
\newcommand{\Ga}{{\mathcal G}}
\newcommand{\Va}{{\mathcal V}}
\newcommand{\Ea}{{\mathcal E}}
\newcommand{\Fa}{{\mathcal F}}
\newcommand{\Ma}{{\mathcal M}}
\newcommand{\Ta}{{\mathcal T}}
\newcommand{\Wa}{{\mathcal W}}

\begin{document}

\title[Ramsey numbers for a triple of cycles]
{The Ramsey numbers \\ for a triple of long cycles}
\author{Agnieszka Figaj}
\address{Institute of Economics and Management,
State Higher Vocational School in Leszno,
ul.~Mickiewicza 5, 64-100 Leszno, Poland}
\email{\tt agnieszka@figaj.pl}
\author{Tomasz {\L}uczak}
\address{Department of Discrete Mathematics, Adam Mickiewicz University,
ul.~Umultowska 87, 61-614 Pozna\'n, Poland}
\email{\tt tomasz@amu.edu.pl}
\date{September 1, 2007}

\abstract
We find the asymptotic value
of the Ramsey number for a triple of  long cycles,
where the lengths of the cycles are large but may have different parity.
\endabstract

\maketitle

\section{Introduction}

If $G_1,G_2,\ldots,G_k$ are graphs, then the Ramsey number $R(G_1,\dots,G_k)$
is the smallest number $N$ such that
each coloring of the edges of the complete graph $K_N$ on $N$ vertices
with $k$ colors leads to a monochromatic copy of $G_i$ in the $i$th color
for some $i$, $1\le i\le k$. The exact value of $R(G_1,\dots,G_k)$
is known only  when all (or most) of $G_i$'s
are either small, or  of a very special kind
(cf.\ Radziszowski's `dynamic survey'~\cite{Rad}).
In this paper we consider the case
in which $k=3$ and all graphs are long cycles.

The value of the Ramsey number $R(C_{m_1},C_{m_2})$ for a pair of cycles
was determined independently by Faudree and Schelp~\cite{FS1}, and
Rosta~\cite{Ros} (see also \cite{KaRos}). A few years later
Erd\H{o}s {\em et al.}~\cite{EFRS1} found
the value of $R(C_{m_1},C_{m_2},C_{m_3})$ and $R(C_{m_1}, C_{m_2},
C_{m_3},C_{m_4})$ when one of the cycles 
is much longer than the others.
Bondy and Erd\H{o}s~\cite{BE} (see also~\cite{Erd})
conjectured that if $m$ is odd, then the value of $R(C_m,C_m,C_m)$ is equal
to $4m-3$.  {\L}uczak~\cite{L} used the Regularity Lemma to show that
an `asymptotic
version' of this conjecture holds, i.e.,  for a large odd $m$
we have $R(C_m, C_m, C_m)=(4+o(1))m$. Then,
Kohayakawa, Simonovits and Skokan~\cite{KSS} employed the Regularity
Lemma to show that Bondy-Erd\H{o}s' conjecture holds for large values of $m$.
The asymptotic value of the Ramsey number for a triple of long even
cycles was found by Figaj and {\L}uczak~\cite{FL}
(see Theorem~\ref{thm:main}(i) below).
A similar result was proved independently by  Gy\'arf\'as~{\em et al.}~\cite{GRSS},
who also found  an exact solution for a closely related problem
of finding the Ramsey number for a triple of long paths of the same  length.

Below we establish
the asymptotic value  of the Ramsey number for a triple of long cycles
in the non-diagonal case. As in the case of the pair of
cycles, it turns out that the value of  $R(C_{m_1},C_{m_2},C_{m_3})$
strongly depends on the parity of $m_i$'s. Thus,
let $\od x$ be the maximum odd number not larger than $x$ and $\ev x$ denote
the maximum even number not larger than $x$. Then the main result
of this paper can be stated as follows.

\begin{thm}\label{thm:main} Let $\al_1,\al_2,\al_3>0$.
\begin{enumerate}
\item
$R(C_{\ev {\al_1 n}},C_{\ev{\al_2 n}},C_{\ev{\al_3 n}})=$

$\quad\quad\quad(0.5\al_1+0.5\al_2+0.5\al_3+0.5\max \{\al_1,\al_2,\al_3\}
+o(1))n$\,,
\item
$R(C_{\ev {\al_1 n}},C_{\ev{\al_2 n}},C_{\od{\al_3 n}})=$

$\quad(\max \{2\al_1 +\al_2,\al_1+2\al_2,0.5\al_1+0.5\al_2+\al_3\}+o(1))n$\,,
\item
$R(C_{\ev{\al_1 n}},C_{\od{\al_2 n}},C_{\od{\al_3 n}})=$

$\quad\quad\quad\quad\quad(\max \{4\al_1,\al_1+2\al_2,\al_1+2\al_3\}+o(1))n$\,,
\item
$R(C_{\od{\al_1 n}},C_{\od{\al_2 n}},C_{\od{\al_3 n}})$
$=(4\max \{\al_1,\al_2,\al_3\}+o(1))n$\,.
\end{enumerate}
\end{thm}

Let us note that the first part of the above theorem is just a reformulation of
the result of Figaj and {\L}uczak~\cite{FL}, while the forth one
follows rather easily from the result of {\L}uczak~\cite{L} (see the end of
Section~\ref{sec:regularity}). Thus, our main task will be to
verify (ii) and~(iii).

The structure of the paper is the following.
In the next section we briefly sketch our approach
to the problem which  follows the idea introduced in~\cite{L}.
It is based on a simple observation that, because of the Regularity Lemma,
finding large monochromatic cycles in
$k$-colored graphs is not much harder than finding matchings
in monochromatic components in $k$-colored graphs. Thus,
in order to show Theorem~\ref{thm:main}, one needs to prove
a Ramsey-like result for matchings.
To this end, in Section~\ref{sec:structure}
we state two  structural results characterizing graphs without
large matchings. In the next section we briefly describe the
way we approach this problem and give a few technical
lemmas needed for our argument.
Finally, in the next two parts of the paper,
Sections~\ref{sec:even} and \ref{sec:odd}, we prove the second and the third
parts of Theorem~\ref{thm:main}, respectively.

All graphs considered in this paper are simple, without loops and multiple edges.
For a graph $G=(V,E)$, and  disjoint subsets $A$, $B$ of~$V$,
by $e(A,B)=e_G(A,B)$ we mean the number of edges
$\{u,v\}$ with $u\in A$ and $v\in B$.
By $G[A]$ we denote  the subgraph
of $G$ induced by $A\subseteq V$.
Throughout the paper  $d(v)=d_G(v)$ stands
for the degree of a vertex $v\in V$
in a graph $G=(V,E)$, while $d(G)=2|E|/|V|$ is
the average degree of (vertices in) a graph $G$.

\section{Cycles and matchings}\label{sec:regularity}

Let us recall that the proof of Theorem~\ref{thm:main}
relies on a simple observation from~\cite{L} that
finding a matching contained in a monochromatic
component in a `cluster graph' covering an $\al$-fraction of vertices
leads to a monochromatic cycle covering $(1+o(1))\al$ vertices in
the original graph. Below  we state   this connection
in a formal way. The main technical nuisance here is that the
`cluster graph' used in the Regularity Lemma is not complete:
it lacks a small fraction of edges which, however, can be made
arbitrarily  small.

In order to make our argument precise we
introduce two relations $\sigma_{t,s}$ and $\tau_{t,s}$ defined in the
following way.  Let $t, s$ be nonnegative integers
and let $a_1,\ldots,a_{t+s},c$
be positive real numbers. Then the relation
$\sigma_{t,s}(a_1,\ldots,a_{t+s};c)$ holds if for every $\delta >0$
there exists $\eps >0$ and $n_0$ such that for every $n>n_0$ and any graph
$G$ on $N>(1+\delta)cn$ vertices, with at least $N>(1-\eps)\binom {N}{2}$
edges, every edge coloring of $G$ using $t+s$ colors results in either
an even cycle of length $\ev {a_i n}$ in the $i$-th color,
for some $i=1,\ldots,t$,
or an odd cycle of length $\od {a_j n}$ in the $j$-th color,
for some $j=t+1,\ldots,t+s$.

In a similar way we define the relation $\tau_{t,s}(a_1,\ldots,a_{t+s};c)$
which holds if for every $\delta >0$
there exists $\eps >0$ and $n_0$ such that for every $n>n_0$ and any graph
$G$ on $N>(1+\delta)cn$ vertices, with at least $N>(1-\eps)\binom {N}{2}$
edges, every edge coloring of $G$ using $t+s$ colors results in  either
a matching contained in a component of the $i$-th color,
for some $i=1,\ldots,t$, saturating at least $a_i n$ vertices,
or a matching contained in a non-bipartite component of the $j$-th color,
for some $j=t+1,\ldots,t+s$, saturating at least
$a_j n$ vertices.

Notice that if $\sigma_{t,s}(a_1,\ldots,a_{t+s}; c)$ holds then
\begin{equation}\label{eqr}
R(C_{\ev{a_1 n}},\dots,C_{\ev{a_t n}},C_{\od{a_{t+1} n}},
\dots, C_{\od{a_{t+s} n}}) \le (c+o(1))n.
\end{equation}
Indeed, inequality (\ref{eqr}) says that the condition
in definition of $\sigma_{t,s}$
holds for $\eps=0$.
The following simple fact is a straightforward consequence of
the definitions of $\sigma_{t,s}$ and $\tau_{t,s}$.

\begin{lemma}\label{eqrr}
\begin{enumerate}
\item If $\sigma_{t,s}(a_1,\dots, a_{t+s}; c)$
then $\tau_{t,s}(a_1,\dots, a_{t+s}; c)$.
\item If $t\le t'$, $t+s=t'+s'$,  $a_i\ge a'_i$ for $i=1,\dots,t+s$, and
$\tau_{t,s}(a_1,\dots, a_{t+s}; c)$ holds,
 then we have also $\tau_{t',s'}(a'_1,\dots, a'_{t+s}; c)$.\qed
\end{enumerate}
\end{lemma}

The next result, crucial for our argument, states that
the relation from Lemma~\ref{eqrr}(i) can be  reversed.

\begin{lemma}\label{lr2}
If $\tau_{t,s}(a_1,\dots, a_{t+s}; c)$
 then $\sigma_{t,s}(a_1,\dots, a_{t+s}; c)$.
\end{lemma}

The proof of Lemma~\ref{lr2} is based on the Szemer{\'e}di's Regularity Lemma.
Let us first recall some definition
related to this result.
Let $G=(V,E)$ be a graph and let $A$, $B$ be disjoint subsets of $V$.
We say that a pair $(A,B)$
is $(\epsilon ,G)$-regular for some $\epsilon > 0$ if for every
$A'\subseteq A$, $|A'|\geq \epsilon |A|$,  and $B'\subseteq B$,
$|B'|\geq \epsilon |B|$, we have
$$\left |{\frac{e(A',B')}{|A'||B'|}}-
{\frac{e(A,B)}{|A||B|}}\right|<\epsilon.$$
A partition $\Pi =(V_i)_{i=0}^{k}$ of the vertex set $V$ of $G$ is
$(\epsilon , k)$-equitable if $|V_0|\leq \epsilon |V|$ and
$|V_1|=\cdots=|V_k|$. An $(\epsilon ,k)$-equitable partition  $\Pi
=(V_i)_{i=0}^{k}$
is $(k, \epsilon , G)$-regular if at most ${\epsilon }\binom{k}{2}$
of the pairs $(V_i, V_j)$, $1\leq i< j\leq k$, are not $(\epsilon, G)$-regular.
Szemer{\'e}di's Regularity Lemma~\cite{S} (see also \cite{KS})
states that every graph $G$
admits an $(k, \epsilon , G)$-regular partition for some $k$,
where $1/\eps\le k\le K_0$, and the constant $K_0$ depends
only on $\eps$ but not on  the choice of $G$.
Below we use the following general version of this result.

\begin{lemma}\label{thm:rl}
For every $\epsilon >0$, $k$, and $\ell$, there exists
$K_0=K_0(\epsilon,k_0,\ell)$ such that the following holds.
For all graphs, $G_1,G_2,\ldots,G_{\ell}$,
with $V(G_1)=V(G_2)=\ldots=V(G_{\ell})=V$ and $|V|\geq k_0$, there
exists a partition $\Pi =(V_0, V_1, \dots, V_k)$ of $V$
such that $k_0\leq k\leq K_0$
and $\Pi $ is $(k, \epsilon , G_r)$-regular
for all $r=1,2,\ldots,\ell$.\hfill\qed
\end{lemma}

We shall also need the following simple property of
$(\epsilon ,G)$-regular pairs (for a similar result see, for instance, \cite{L}).

\begin{lemma}\label{wrl} Let $T\ge 2$, $0<\eps<1/(100T)$, and
let $G=(V,E)$ be a bipartite graph with bipartition $\{V_1,V_2\}$ such
that $|V_1|=|V_2|=n>10T\eps^{-2}$. Furthermore, let $e(V_1,V_2)\geq |V_1||V_2|/T$
and let the pair $(V_1,V_2)$ be $(\eps ,G)$-regular.
Then, for every $\ell$, $1\leq \ell\leq n-5\eps n$, and every pair of
vertices $v'\in V_1$, $v''\in V_2$, where $d(v'),d(v'')\geq n/(5T)$, $G$
contains a path of length $2\ell+1$ connecting $v'$ and $v''$.\hfill\qed
\end{lemma}

Now we can show Lemma~\ref{lr2}.

\begin{proof}[Proof of Lemma~\ref{lr2}]
Let $a_i\ge a'_i$ for all $i=1,\dots,t+s$ and
let $G$ be a graph with $N>(1+\delta) c n$ vertices and at least
$(1-\eps_{\tau}^4(\del/2))\binom {N}{2}$ edges, where
$\eps_{\tau}(\del/2)$ is a constant defined as in relation
$\tau_{t,s}(a_1,\dots, a_{t+s}; c)$. Let us assume that
$(G_1,G_2,\ldots,G_{t+s})$ is a $t+s$-coloring of the edges of $G$.

Now let $\eps = \min \{\del/4, \eps_{\tau}^4(\del/2)\}$.
Apply Lemma~\ref{thm:rl} to find a partition
$\Pi=\{V_0,V_1,\ldots,V_k\}$ of vertices of $G$ such that
$1/\eps_{\tau}(\del/2)\leq k\leq K_0\,$
and $\Pi$ is $(k, \epsilon , G_{\ell})$-regular for all $\ell=1,2,\ldots,t+s$.

Let $\Ga=(\Va, \Ea)$ be the graph with vertex set $\Va=\{V_1, V_2, \dots,V_k\}$
and
\begin{multline*}
\Ea=\{\{V_i, V_j\}: (V_i, V_j)\textrm{\ is $(\eps , G_{\ell})$-regular
for}\  \ell=1, 2,\ldots,t+s,\\
\textrm{and\ } d_G(V_i,V_j)\ge \frac 12|V_i||V_j|\}\,.
\end{multline*}
Then $|\Ea|\geq (1-\eps_{\tau}^2(\del/2))\binom{k}{2}$.
Construct a $t+s$-coloring $(\Ga_1,\Ga_2,\ldots, \Ga_{t+s})$ of $\Ea$
by coloring an edge $\{V_i, V_j\}$ with the lexicographically first color $\ell$
for which
$$e_{G_{\ell}}(V_i, V_j)\geq \frac{|V_i||V_j|}{2(t+s)}\,.$$
Let $k=(1+\del/2)ck'$. Then there is a color $\ell_0$, $1\le \ell_0\le t+s$,
such that $\Ga$ contains a matching $\Ma=\{e_1,e_2,\dots,e_q\}$
saturating  $2q\ge a_{\ell_0}k'$
vertices of $G$ in the monochromatic component of $\Ga$ in the $\ell_0$th color;
furthermore, if $\ell_0\ge t+1$, then this component is non-bipartite.
Note that since for every $i$, $i=1,2,\ldots,k$, we have
$$|V_i|\ge \frac{(1+\del)cn-\eps n}{k}=\frac{(1+\del)cn-\eps n}{(1+\del/2)ck'}
\ge (1+\del/4)\frac{n}{k'},$$
the total number of vertices of $G$ contained in the sets $V_i$ saturated by $\Ma$
is at least
\begin{equation}\label{eq:1}
a_{\ell_0}k'|V_i|\ge a_{\ell_0}k'(1+\del/4)\frac{n}{k'}=(1+\del/4)a_{\ell_0}n
\ge (1+6\eps)a_{\ell_0}n\,.
\end{equation}
We shall show that the subgraph spanned in $G$ by these vertices contains a
monochromatic cycle in the $\ell_0$th color on precisely
$\ev {a_{\ell_0}n} $ vertices if $1\le \ell_0\le t$,
and with $\od {a_{\ell_0}n} $ vertices if $t+1\le \ell_0\le t+s$.

Let us assume first  $t+1\le \ell_0\le t+s$.
Let us recall that $\Ga$
contains a monochromatic component $\Fa$ in the $\ell_0$th color which contains
a matching  $\Ma=\{e_1,e_2,\dots,e_q\}$, $2q\ge a_{\ell_0}k'$, and an
odd cycle $\Ca=W_1W_2\dots W_{p'}\hV_1$. Observe that $\Fa$ contains
a closed walk $\Wa=\hV_1\hV_2\dots \hV_{p}\hV_1$ of
an odd length, containing all edges of $\Ma$. Indeed, it is easy to see that
the minimum connected subgraph of $\Fa$ which contains all edges of $\Ma$ and a vertex
$\hV_1$ of $\Ca$ is a tree $\Ta$. Since clearly there is an (even) walk
$\Wa'$ which traverses each edge of $\Ta$ precisely two times,
in order to get $\Wa$ it is enough
to enlarge $\Wa'$ by edges of $\Ca$.

Now, using elementary properties of $(\eps,G)$-regular pairs, it is easy to
find in $G$ an odd cycle $C=\hv_1\hv_2\dots \hv_p\hv_1$ such that
$\hv_i\in \hV_i$ for $i=1,\dots,p$, and whenever the pair $(\hV_i,\hV_{i+1})$
belongs to a matching $\Ma$ then both
$d_{G_{\ell_0}}(\{\hv_i\},\hV_{i+1})\ge
\frac{|\hV_{i+1}|}{4(t+s)}\,,$
 and
$d_{G_{\ell_0}}(\hV_i,\{\hv_{i+1}\})\ge
\frac{|\hV_{i}|}{4(t+s)}\,.$
Now we can apply Lemma~\ref{wrl} and replace all edges $\{\hv_i,\hv_{i+1}\}$,
such that $\{\hV_i,\hV_{i+1}\}$ belongs to $\Ma$, by long paths not containing
any other vertices of $C$. Since, by (\ref{eq:1}), the number of vertices
of $G$ contained in $V_i$'s saturated by $\Ma$ is larger than
$(1+6\eps)a_{\ell_0} n$,  we can do it in such a way that
the resulting monochromatic cycle in the $\ell_0$th color has length
$\od{a_{\ell_0} n}$.

If $1\le \ell_0\le t$ then the argument is basically the same.
Here we start with a closed walk $\Wa$ in $\Ga$ of an even length
which contains all edges of $\Ma$ and, based on $\Wa$, we construct an even
cycle $C$ in $G$. Then, as in the previous case, one can
use Lemma~\ref{wrl} to enlarge $C$ to a cycle of length $\ev{a_{\ell_0} n}$.
\end{proof}

{From} Lemma~\ref{eqrr} and~\ref{lr2} we get the following corollary.

\begin{cor}\label{wn}
If $t\le t'$, $t+s=t'+s'$ and $a_i\ge a'_i$ for $i=1,\dots,t+s$ and
$\sigma_{t,s}(a_1,\dots, a_{t+s}; c)$ holds then we have
$\sigma_{t',s'}(a'_1,\dots, a'_{t+s}; c)$.\qed
\end{cor}

Let us comment that, unlike the relation $\sigma_{t,s}(a_1,\dots,a_{t+s})$,
we do not know any simple proof that the Ramsey number for cycles
is monotone, e.g., although most certainly we  have
$$R(C_{20},C_{20},C_{20})\leq R(C_{22},C_{20},C_{20})
\leq R(C_{21},C_{20},C_{20})$$
it is by no means clear how to verify it directly (i.e., without estimating
the Ramsey numbers above).
However, using Corollary~\ref{wn}, we can deduce Theorem~\ref{thm:main}(iv)
from {\L}uczak's result on $R(C_n,C_n,C_n)$.

\begin{proof}[Proof of Theorem~\ref{thm:main}(iv)] Let us assume that
$\al_1\ge \al_2\ge \al_3>0$. {\L}uczak~\cite{L} showed that
$\sigma_{0,3}(\al_1,\al_1,\al_1;4\al_1)$ holds.
Thus, by Corollary~\ref{wn},
$\sigma_{0,3}(\al_1,\al_2,\al_3;4\al_1)$ holds as well and, consequently,
$$R(C_{\od{\al_1 n}},C_{\od{\al_2 n}},C_{\od{\al_3 n}})\le(4\al_1+o(1))n\,.$$
In order to show the lower bound for
$R(C_{\od{\al_1 n}},C_{\od{\al_2 n}},C_{\od{\al_3 n}})$
consider the following coloring of the complete graph $K_N$
on $N=4\od{\al_1 n}-4$ vertices. Split the vertices of $K_N$ into
four equal parts $V_1$, $V_2$, $V_3$, and $V_4$.
Color the edges inside each of
$V_i$'s with the first color, the edges in pairs $(V_1,V_2)$,
$(V_2,V_3)$, $(V_3, V_4)$ with the second one, and the edges in pairs
$(V_1,V_3)$, $(V_2,V_4)$, $(V_1,V_4)$ with the third color.
Clearly in this coloring we have no monochromatic cycles  longer than
$\od{\al_1 n}-1$ in the first color, and no odd cycles in
either the second or third color. Hence,
$$R(C_{\od{\al_1 n}},C_{\od{\al_2 n}},C_{\od{\al_3 n}})
\ge 4\od{\al_1 n}-3\,,$$
and Theorem~\ref{thm:main}(iv) follows.
\end{proof}

\section{Two structural results}\label{sec:structure}

This short section consists of two  simple results on the structure
of graphs without large matchings.

Let us start with the following consequence of Tutte's
theorem observed in~\cite{FL}. Since it is crucial
for our approach we recall its proof here for the completeness
of the argument.

\begin{lemma}\label{tutte}
If a graph $G=(V,E)$
contains no matchings saturating at least $n$ vertices, then
there exists a partition $\{S,T,U\}$ of $V$ such that:
\begin{enumerate}
\item the subgraph induced in $G$ by $T$ has maximum degree
at most $\sqrt{|V|}-1$,
\item there are no edges between the sets $T$ and $U$,
\item $|U|+2|S|<n+\sqrt{|V|}$.
\end{enumerate}
\end{lemma}

\begin{proof} From Tutte's theorem, if a graph $G=(V,E)$ contains no matchings
saturating at least $n$ vertices, then there exists a subset $S$
such that the number of odd components in a graph $G[V\setminus S]$
is larger than $|V|+|S|-n$. Split the set of these components into two parts:
those with at most $\sqrt{|V|}$ vertices and those larger than $\sqrt {|V|}$.
The set of vertices which belong to the components from the former family
we denote by $T$,
the set of vertices of the component from the latter one by $U$.
Then, for such a partition $V=S\cup T\cup U$, (i) and (ii) clearly hold.
Moreover,
since there are fewer than $\sqrt {|V|}$ components larger than $\sqrt{|V|}$,
Tutte's condition gives
$$|T|>|V|+|S|-n-\sqrt{|V|}\,,$$
so that
$$|U|=|V|-|S|-|T|< n+ \sqrt{|V|}-2|S|\,,$$
which gives (iii).
\end{proof}

Graphs without a large matching contained in a non-bipartite
component  have a rather simple characterization as well
(cf. {\L}uczak~\cite{L}). Let us recall first a classical result of Erd\H{o}s and
Gallai~\cite{EG}.

\begin{lemma}\label{EG}
Each graph with $n$ vertices and at least $(m-1)(n-1)/2+1$ edges,
where $3\leq m\leq n$, contains a cycle of length at least $m$.
In particular, it contains a component with a matching saturating
at least $m-1$ vertices.\qed
\end{lemma}

Now our second structural lemma can be stated as follows.

\begin{lemma}\label{l:nonbipartite}
If no non-bipartite component of a graph $G=(V,E)$ on $n$ vertices
contains a matching saturating at least $\al n$ vertices, then
there exists a partition $V=V'\cup V''$ of $V$ such that:
\begin{enumerate}
\item $G$ contains no edges between sets $V'$ and $V''$,
\item the graph $G'=G[V']$ induced in $G$ by $V'$ is bipartite,
\item the graph $G''=G[V'']$ induced in $G$ by $V''$ contains
at most  $0.5\al n|V(G'')|$ edges.
\end{enumerate}
\end{lemma}

\begin{proof}
Denote by $H_1,\dots,H_r$ components of  $G$.
Let $V'$ consist of the vertices of all the components which are
bipartite and $V''=V\setminus V'$.
Then, (i) and (ii) clearly hold. Note also that
if any non-bipartite component $H_i$ has average degree larger than
$\al n$, then,
by Erd\H{o}s-Gallai theorem, it contains  a cycle longer than $\al n$
and thus also a matching of size at least $\al n$ contradicting our assumption.
Thus, every component of $G''$ has average degree at most $\al n$
and (iii) follows.
\end{proof}

\section{The first look at the matching problem}

Let us recall that, in order to show Theorem~\ref{thm:main},
it is enough to verify the property $\tau_{t,s}(\al_1,\al_2,\al_3;c)$
for appropriately chosen $t$, $s$, $t+s=3$, and $c=c(\al_1,\al_2\al_3$,
i.e., we need to prove that
in every sufficiently large three-colored `nearly complete' graph $G$ we can find
a monochromatic component containing a large matching. Lemmas~\ref{tutte}
and~\ref{l:nonbipartite} suggest the following approach. Suppose that
a component $F=(V_F,E_F)$ of the subgraph $G_1$ of $G$ induced by the first color
contains no large matchings. Then, using Lemma~\ref{tutte},
one can decompose the set  vertices $V_F$ of $F$ into three sets $S$, $T$, $U$.
Delete from $F$ all vertices from $S$. Then in the remaining graph $H=F[T\cup U]$,
all edges joining $T$ and $U$, as well as all but a negligible fraction of
edges contained in $T$, are colored with either the second or the third color.
Consequently, to study matchings in three-colored `nearly complete' graphs,
one should first study  matchings in two-colored `nearly complete' graphs
with `holes' (in our case the hole is the set $U$).

Lemma~\ref{l:nonbipartite} suggests a similar approach.
Suppose that in a 'nearly complete'
three-colored $G=(V,E)$ on $n$ vertices the graph $G_1$ induced by the first color
has average degree $d(G_1)=\rho n >\al_1 n$ yet it contains no non-bipartite
component with a  matching saturating at least $\al_1 n$ vertices.
Then, there is a partition of the set of vertices of $G_3$
in the form $V=W_1\cup W_2\cup R$, where $\{W_1,W_2\}$ is a bipartition of $G'$,
$R$ is the set of vertices of $G''$, and $G'$ and $G''$ are the graphs
described in  Lemma~\ref{l:nonbipartite}. Note that since $d(G'')\le \al_1 n$,
so we must have $d(G')\ge \rho n$, and so the larger of the sets $W_1$,
$W_2$, say $W_1$, must have at least $\rho (n)$ vertices and
$|R|\le (1-2\rho) n$. Thus, in this case, the graph $F=G[W_1\cup R]$
is a 'nearly complete' graph with nearly all of the edges,
except those contained in the hole $R$,  colored with just two colors.

 Since our proof of Theorem~\ref{thm:main} is based on the above idea
here we state two results, Lemmas~\ref{l:dwa} and \ref{l:trzy},
which determine the size of the largest
matchings when the edges of a  `nearly complete graph with a hole'
are colored with two colors. We begin however with two
technical results from~\cite{FL} we state  without proofs: the first
one characterizes   matchings in a `nearly complete' bipartite graphs,
the second one describes matchings in `nearly complete' bipartite graphs with
two holes.

\begin{lemma}\label{l2}
Let $G=(V,E)$ be a bipartite graph with bipartition
$\{V_1,V_2\}$, where $|V_1|\ge|V_2|$,
and at least $(1-{\eps})|V_1||V_2|$ edges, where
$0<{\eps}<0.01$. Then there is  a component in $G$
of at least $(1-3\eps)(|V_1|+|V_2|)$ vertices which contains
a matching of cardinality at least $(1-3{\eps})|V_2|$.\qed
\end{lemma}

\begin{lemma}\label{double}
Let $0\le \nu_1\le \nu_2\le 1$, $0<\eps<0.01\nu_1$, $N\ge 4/\eps$,
and let $U_1$, $U_2$ be two, not necessarily disjoint,
subsets of $[N]=\{1,2,\dots, N\}$ of $\nu_1 N$ and $\nu_2 N$
vertices respectively.
Let $G=([N],E)$ be a graph obtained from the complete graph
on the set $[N]$ vertices by removing all edges
contained in  $U_i$, $i=1,2$, and, possibly, at most $\eps^3 \binom {N}{2}$
other edges. Then $G$ contains a component with a matching saturating at least:
\begin{enumerate}
\item $(1-5\eps)N$ vertices if $|U_2|\le N/2$;
\item $(2-7\eps)N-2|U_2|$ vertices if $|U_2|\ge N/2$.\qed
\end{enumerate}
\end{lemma}

Before we state and prove two main results of this section on
matchings in two-colored `nearly complete' graphs with holes, let us
make a simple observation we shall  often use in the proof.
Suppose that a graph $G_W=(V,E)$ is obtained from the complete graph
with vertex set $V$ by removing all edges contained in $W\subseteq V$.
Let us color edges of $G_W$ by two colors, and let $G_1$,
$G_2$ be spanning subgraphs of $G$ induced by the first
and the second color respectively. Then, either one
of these graphs is connected, or there is a partition of
$W=W_1\cup W_2$ into two non-empty sets $W_1$, $W_2$ such that all
edges with one end in $W_i$, $i=1,2$, are colored with the $i$th color.

\begin{lemma}\label{l:dwa}
For every $\al, \be>0$, $ \nu \ge 0$, $\max \{\al,\be,\nu\}=1$, and
$0<\eps <0.01\min \{\al,\be\}$, there exists $n_0$, such that
for every  $n>n_0$ the following holds.

Let $G=(V,E)$ be a graph obtained from the complete graph
on
$$N=(0.5\al +0.5\be +\max \{\nu ,0.5\al,0.5\be\}+3\sqrt{\eps})n$$
vertices by removing all edges contained in a subset $W\subseteq V$ of size
 $\nu N$ and no more than  $\eps ^3n^2$ other edges.
Then, every coloring of the edges of $G$ with two colors leads to
either a monochromatic component colored with the first color
containing a matching saturating at least   $(\al +\eps)n$
vertices, or  a monochromatic component of the second color
containing a matching saturating at least  $(\be +\eps)n$
vertices.
\end{lemma}

\begin{proof}
Let us consider a two-coloring of  edges of a graph $G$ which
fulfills the assumption of the lemma,
and denote graphs induced by the  edges of the first
and the second colors by $G_1$ and $G_2$  respectively.
Let $F$ denote the largest monochromatic component
in this coloring. Without loss of generality we can assume that
$F$ is colored with the first color. We consider two following cases.

\medskip

{\sl Case 1.} $|F|\ge N-\sqrt{\eps} n$.

\smallskip

Let us assume that $F$ contains no matching saturating $(\al+\eps)$
vertices. Then one can use  Lemma~\ref{tutte} to find a
partition of the set of vertices of $F$ into sets $S,T,U$
such that there are no edges of the first color between $T$ and $U$,
there are at most $\sqrt N |T|$ edges of the first color contained in $T$,
and furthermore
\begin{equation}\label{eqd1}
2|S|+|U|\le \al n+\eps n+\sqrt{N}.
\end{equation}
Now let us consider the graph $G'=G_2[T\cup U]$. We shall show that
it contains a component with a matching saturating at least
$(\be+\eps)n$  vertices. Since $G'$ is a `nearly complete' graph
on
$|T|+|U|\ge N-\sqrt{\eps} N-|S|\ge  N-0.5{\al n}-2\sqrt{\eps}n$.
with two holes, $U$ and $W$, we apply Lemma~\ref{double}.
Thus, if $|W|,|U|\le (N-|S|)/2$, Lemma~\ref{double}(i) implies
that there exists a component of the second color with a
matching saturating at least
\begin{equation}\label{eqd44}
\begin{aligned}
N-0.5{\al n}&-2\sqrt{\eps}n-5\eps N\\
&\ge 0.5\be n +\max \{\nu ,0.5\al,0.5\be\}+3\sqrt{\eps}n- 5\eps n-2\sqrt{\eps}n\\
&\ge \be n+\eps n,
\end{aligned}
\end{equation}
vertices. In the case in which $\max\{|U|,|W|\}\ge  (N-|S|)/2$, from
Lemma~\ref{double}(ii) we infer that
there exists a component of the second color with a
matching saturating at least
\begin{equation}\label{eqd2}
\begin{aligned}
2N - 4\sqrt{\eps} n - 7\eps N  -2 & |S|- 2 \max\{|U|,|W|\}\\
&\ge 2N-2|S|-2\max\{|U|,|W|\}-5\sqrt{\eps} N
\end{aligned}
\end{equation}
vertices. Thus, if $|W|\ge |U|$, then using (\ref{eqd1})
we can estimate the right hand side of (\ref{eqd2})
by
\begin{equation}\label{eqd2a}
\begin{aligned}
2N-2|S|-2|W|-5\sqrt{\eps} N&\ge 2N-\al n-\eps n-\sqrt{N}-2\nu N - 5{\eps}N
\\&\ge \be n+\sqrt{\eps}n-2\eps N>\be n +\eps n,
\end{aligned}
\end{equation}
while for $|U|\ge |W|$ (\ref{eqd1}) gives
\begin{equation}\label{eqd3}
\begin{aligned}
2N-2|S|-2|U|-5\sqrt{\eps} N&\ge 2N-2\al n-2\eps n-2\sqrt{N}-5\eps N
\\&\ge \be n+\sqrt{\eps} n-6\eps N>\be n+\eps n.
\end{aligned}
\end{equation}
This completes the proof in this case.

\medskip

{\sl Case 2.} $|F|< N-\sqrt{\eps} n$.

\smallskip

As we have already noticed in the remark preceding the statement of
the lemma, every two-coloring which does not lead to a large monochromatic
component must have a rather special structure.  Thus, let us
denote by  $W_1$ the set of vertices $w_1$ of $W$ such that all but
at most $\eps n$  edges adjacent to $w_1$ are colored with the first color,
by   $W_2$ the set of vertices $w_2$ of $W$ such that all but
at most $\eps n$  edges adjacent to $w_2$ are colored with the second color,
and $W_0=W\setminus (W_1\cup W_2)$. Since in the graph $G$ lacks at most
$\eps^3 n^2$ edges joining $W$ with $V\setminus W$ we must have $|W_0|\le \eps n$.
Furthermore, $|F|\le N-\sqrt{\eps} n$ implies that
$\max\{|W_1|,|W_2|\}\ge 0.5\sqrt{\eps}n$.

Let us set $|W_1|=\al' n$, $|W_2|=\be' n$. Note that $|V\setminus W|\ge
\max\{0.5\al,0.5\be\}+3\sqrt{\eps}n$, so if either $\al'\ge 0.5\al+7\eps$
or $\be' \ge 0.5\be +7\eps$, then we are done by Lemma~\ref{l2}.
More generally, if the graph $H=G[V\setminus W]$ contains either a
monochromatic component in the first color with a matching saturating
at least $\al''n=\al n  - 2\al'n+15\eps n$ vertices, or   a
monochromatic component in the second color with a matching saturating
at least $\be''n=\be n  - 2\be'n+15\eps n$ vertices, the assertion
follows as well.
Indeed, observe first that
because  $|W_1|, |W_2|> 0$, all vertices of $H$ except at most $2\eps n$ belong to the component
of the first color and, in the same way, there are at most $2\eps n$
vertices of $H$ which do not belong to the large component of the second color.
Thus, we can first find a large monochromatic matching in the large
component of the $i$th color, $i=1,2$, and then match unsaturated vertices
of this component to vertices of $W_i$ using Lemma~\ref{l2}.

Now note that every two coloring of a `nearly complete'
graph leads to a large monochromatic component in one of the colors,
so the Case~1 considered above covers all cases in which $\nu=0$.
Furthermore,
\begin{equation*}\label{eqd3d}
\begin{aligned}
0.5\al''+&0.5\be''+\max\{0.5\al'',0.5\be''\}+3\sqrt{\eps}\\
&\le 0.5\al+0.5\be+\max\{0.5\al,0.5\be\}-0.5\sqrt{\eps} -\nu +20\eps+3\sqrt{\eps}\\
&\le 0.5\al+0.5\be+\max\{\nu, 0.5\al,0.5\be\}-\nu +2.9\sqrt{\eps}<N-\nu\,.
\end{aligned}
\end{equation*}
Thus,  Case 1 we have just proved implies that
$H$ contains either a large component in the
first color with a matching saturating at least $\al''n$ vertices, or   a
monochromatic component in the second color with a matching saturating
at least $\be''n$ vertices, and the assertion follows.
\end{proof}

If we wish to have one of the matching in a non-bipartite monochromatic
component, then the condition becomes slightly more complicated.

\begin{lemma}\label{l:trzy}
Let $\al, \be>0$, $\nu \ge 0$, $\max \{\al,\be,\nu\}=1$,
$0<\eps <0.01\min \{\al,\be\}$, and let
\begin{equation}\label{xi}
\begin{aligned}
\xi=\xi (\al, \be, \nu )=\max \Big\{0.5\al+0.5\be+\max &\{0.5\al,0.5\be,\nu\},\\
&1.5\al+\max \{0.5\al,\nu\}\Big\}.
\end{aligned}
\end{equation}
Then,  there exists $n_0$, such that
for every  $n>n_0$ the following holds.

Let $G=(V,E)$ be a graph obtained from the complete graph
on $N=(\xi+5\sqrt{\eps})n$
vertices by removing all edges contained in a subset $W\subseteq V$
of size $\nu N$ and no more than  $\eps ^3n^2$ other edges.
Then, every coloring of the edges of $G$ with two colors leads to
either a monochromatic component colored with the first color
containing a matching saturating at least   $(\al +\eps)n$
vertices, or  a non-bipartite monochromatic component of the second color
containing a matching saturating at least  $(\be +\eps)n$
vertices.
\end{lemma}

\begin{proof}
Consider a two-coloring of edges of a graph
 $G=(V,E)$ which fulfills the assumption of the lemma
 and let $G_i$, $i=1,2$, denote the graph spanned by edges of the $i$th color.
 Since
\begin{equation*}
\xi (\al, \be, \nu )\geq 0.5\al+0.5\be+\max \{0.5\al,0.5\be,\nu\},
\end{equation*}
from  Lemma~\ref{l:dwa}  it follows that either there exists a
component of $G_1$ which contains a matching saturating at least
$(\al+\eps)n$ vertices, or there exists a component $F_2$ in $G_2$
which contains a matching saturating at least $(\be +\eps)n$ vertices.
Thus, the assertion follows unless the component $F_2$ is bipartite.
Hence, we shall assume that $F_2$ is bipartite with bipartition
$\{Z_1,Z_2\}$ and split the proof into the
following two cases.

\medskip

{\sl Case 1.} $|F_2|\ge N-\sqrt{\eps} n$.

\smallskip

Since in this case
$$|Z_1|+|Z_2| \ge (1.5\al+\max \{0.5\al,\nu\}+4\sqrt{\eps})n\,,$$
for some $i_0=1,2$, we have both $|Z_{i_0}|\ge (\al+2\sqrt{\eps})n$
and $|Z_{i_0}\setminus W|\ge (0.5\al+\sqrt{\eps})n$. But then,
due to Lemma~\ref{l2}, the graph $G_1[Z_{i_0}]$ contains a component
with a matching saturating at least $(\al+\eps)n$ vertices.

\medskip

{\sl Case 2.} $|F_2|< N-\sqrt{\eps} n$.

\smallskip

Let $F_1$ denote the largest component of $G_1$. Let us consider
first the case when $|F_1|<N-\sqrt{\eps}n$. Then, by the remark
preceding the statement of Lemma~\ref{l:dwa}, all but at most $\sqrt{\eps}n$
vertices of $V\setminus W$ are contained in one of the sets of the bipartition
of $F_2$, say, in $Z_1$. But then all edges of $G$ with both ends
in $Z_1$ are colored with the first color and $|Z_1|\ge (3\al/2+\sqrt{\eps})n$.
Consequently, by Lemma~\ref{l2}, there is a component in $G_1$
with a matching saturating at least $(\al n+\eps)n $ vertices and the
assertion follows.

Thus, we may and shall assume that the largest component $F_1=(V_1,E_1)$
of $G_1$
has at least $N-\sqrt{\eps}n$ vertices. Suppose that it contains no matchings
saturating at least $(\al+\eps)n$ vertices. Let $S,T,U$ be the sets whose
existence is assured by Lemma~\ref{tutte}; in particular we have
\begin{equation}\label{eqt1}
2|S|+|U|\le \al n+\eps n+\sqrt{N}.
\end{equation}
Note that the subgraph $H_2=G_2[T\cup U]$ contains a component
$F'_2$ with a matching saturating at least $(\be+1.1\eps)n$ vertices.
Indeed, consider graph $\hat G$
with the same set of vertices as $F_1$, obtained from $G$ by deleting
all edges of $G_1$ with both ends in $T$. Then $\hat G$ fulfills assumptions
of Lemma~\ref{l:dwa} with $\eps'=1.1\eps$. On the other hand, if we color
with the first color  all edges of $\hat G$ which have either one end in $S$,
or both ends in $U$, we create no matching in this color  saturating more than
$(\al+1.1\eps)n$ vertices. Hence, by Lemma~\ref{l:dwa},
there must be component $F'_2$ in
the second color  which contains a matching saturating at least $(\be+1.1\eps)n$
vertices, and, because of our construction, $F'_2\subseteq H_2$.

If $F'_2$ is non-bipartite we are done, so let us assume that $F'_2$ is bipartite.
Since $H_2$ contains all edges of $G$ joining $T$ and $U$ and all but
$|T|\sqrt{N}$ edges contained in $T$, it is easy to see that if a graph
$G[T]$ contains a component of size, say, $10\eps n$, then all but at most
$\sqrt{\eps}n$ vertices of $H_2$ lie in the same giant component which
clearly is not bipartite. Thus, because of Lemma~\ref{l2}, in order to
keep $F'_2$ bipartite the set $T$ cannot be much larger than the hole $W$,
i.e.
\begin{equation}\label{eqt1b}
|T|\le |W|+10 \eps n\le (\nu+10\eps)n\,.
\end{equation}
However, from (\ref{eqt1}) and (\ref{eqt1b}) it follows that
\begin{align*}
|F_1|&=|T|+|U|+|S|\leq (\al+\eps +\nu+10\eps)n+\sqrt{N}\\
&\le (\al+\nu+2 \sqrt{\eps})n<(1.5\al+\nu + 3\sqrt{\eps})n\\
&< N-\sqrt{\eps}n\,,
\end{align*}
while as, we have seen,  $|F_1|>N-\sqrt{\eps}n$.
Thus, the component $F'_2$ is non-bipartite and the assertion follows.
\end{proof}

{\sl Remark.} It is easy to construct colorings which shows that the estimates
given by Lemmas~\ref{l:dwa} and \ref{l:trzy} are, up to epsilon terms,
best possible.

\section{Triple of  cycles: one odd, two even}\label{sec:even}

In this part of the paper we prove Theorem~\ref{thm:main}(ii),
i.e., we estimate the Ramsey number for three long cycles,
in which one is odd and the other two have even length.
Let us start however with the following consequence of Lemma~\ref{l:dwa}.

\begin{lemma}\label{f1}
Let $\al_1 \geq \al_2>0$, $0<\eps<0.01\al_2$ and let
$G=(V,E)$ be a graph obtained from the complete graph on
$|V|\ge (2\al_1+\al_2+9\seps)n$ vertices by deleting at most
$|E|\leq \eps^4 n^2$ of its edges.
Then there exists $n_0$ such that for every $n\ge n_0$ the following holds.

Let us suppose that the edges of $G$ are colored with three colors
which spans graphs $G_1$, $G_2$, and $G_3$, and that the graph $G'$
which is a union of the bipartite components of $G_3$ has at least
 $(1.5\al_1+0.5\al_2+8\seps)n$ vertices.
Then, there exists either  a monochromatic component
of the first color which contains a matching
saturating at least $(\al_1 +\eps)n$ vertices,
or a monochromatic component
of the second color with a matching saturating at least
 $(\al_2 +\eps)n$ vertices.
\end{lemma}

\begin{proof} Observe first that it is enough to prove the lemma
for `nearly complete' graphs  $G=(V,E)$ with precisely
$|V|= (2\al_1+\al_2+5\seps)n$ vertices.
Let $G_1$, $G_2$, $G_3$ denote the graphs spanned in $G$ by the first, the
second, and the third color,  respectively. Furthermore,
let $G'$ denote the bipartite graph with bipartition $\{X,Y\}$,
where $|X|\ge |Y|$,
which is the union of all bipartite components of $G_3$.
Let us consider a subgraph $H=(\hat V, \hat E)$ of $G$ whose vertex set is the set
$V\setminus Y$ and edges are all edges of $G$ which are colored with either
the first or the second color. Thus, $H$ is a `nearly complete' graph
on $|\hat V|=|V\setminus Y|$ vertices with a hole $W=V\setminus(X\cup Y)$
of size $|W|=\nu n$. Thus, to complete the proof we need
to verify if the assumptions of  Lemma~\ref{l:dwa} hold for two-colored~$H$.
 To this end note  that
$$\nu\le 0.5\al_1+0.5\al_2+\seps\,,$$
and so
\begin{equation*}\label{eq6}
0.5\nu + 0.5\al_1 \ge \nu -0.5\seps\,.
\end{equation*}
Thus, for the number of vertices of $H$ we get
\begin{align*}
|\hV|&\ge\frac{(2\al_1+\al_2+9\seps)n-\nu n}{2}+\nu n\\
&\geq (0.5\al_1+0.5\al_2+0.5\nu+0.5\al_1 +4.5\seps)n\\
&\geq (0.5\al_1+0.5\al_2+\max\{\nu,0.5\al_1,0.5\al_2\} +4\seps)n\,,
\end{align*}
and the assertion follows from  Lemma~\ref{l:dwa}.
\end{proof}

Now we can find the asymptotic value of
$R(C_{\ev {\al_1 n}},C_{\ev{\al_2 n}},C_{\od{\al_3 n}})$.

\begin{proof}[Proof of Theorem~\ref{thm:main}(ii)]
Let us first estimate
$R(C_{\ev {\al_1 n}},C_{\ev{\al_2 n}},C_{\od{\al_3 n}})$
from above.
{From} Lemma~\ref{lr2} it follows that to prove the upper bound
in Theorem~\ref{thm:main}(ii)  it is enough to
verify that for $\al_1,\al_2,\al_3>0$, $\al_1\ge \al_2$,  and
$c=\max \{2\al_1 +\al_2,0.5\al_1+0.5\al_2+\al_3\}$,
we have  $\tau_{2,1}(\al_1,\al_2,\al_3;c)$ holds.
Observe that we may and shall assume that $\max\{\al_1,\al_2,\al_3\}=1$.

Let $\eps >0$ and let $G=(V,E)$ be a graph
obtained from the complete graph on  $N=(c+10\seps)n$ vertices by removing
at most $|E|\ge \eps^5 n^2$ edges. Let us color edges of $G$ with
three colors and denote the subgraph spanned by the $i$th color by $G_i$,
$i=1,2,3$. Let us consider the two following cases.

\medskip

{\sl Case 1.} $\al_1\ge\al_3$.

\smallskip

Then the number of vertices in $G$ is $|V|=(2\al_1+\al_2+10\seps)$.
If the average density $d(G_3)$ of $G_3$ is smaller than $(\al_1+8\seps)n$,
then either $d(G_1)\ge (\al_1+\eps)n$, or $d(G_2)\ge (\al_2+\eps)n$
and the assertion follows from Lemma~\ref{EG}. Thus, let us consider the case
$d(G_3)>(\al_1+8\seps)n$. Assume that $G_3$ contains no component
with a matching saturated as least $\al_3 n$ vertices and
let $G'$ and $G''$ be two subgraphs of $G_3$ whose existence is
assured by Lemma~\ref{l:nonbipartite}. Note that
$$d(G'')\le \al_3 n < (\al_1+8\seps)n <d(G_3)$$
and so
\begin{equation}\label{eq5}
d(G')\ge d(G_3)> (\al_1+8\seps)n\,.
\end{equation}
Since each bipartite graph with the average degree $m$ must have
at least $2m$ vertices, from (\ref{eq5})  we infer that
$G'$ contains at least
$$2\al_1+16\seps> 1.5\al_1+0.5\al_2+8\seps $$
vertices. Now the existence of a monochromatic component with large
matching in one of the first two colors follows from
Lemma~\ref{f1}.

\medskip

{\sl Case 2.} $\al_1\le \al_3$.

\smallskip

In this case $G$ has  $|V|=(0.5\al_1+0.5\al_2+\al_3+10\seps)n$
vertices. Figaj and {\L}uczak~\cite{FL} (cf. Theorem~\ref{thm:main}(i))
showed that then
either there exists $i_0$, $i_0=1,2$, such that
$G_{i_0}$ contains a monochromatic component in the $i_0$th
color with a matching saturating
at least $\al_{i_0} n$ vertices, so we are done, or there
exists a component $G'$ in the third color which contains a
matching saturating at least
\begin{align*}
|V|-(0.5\al_1+0.5\al_2+\seps)n&\ge (\max\{1.5\al_1+0.5\al_3,\al_3\}+9\seps) n \\
&\ge (1.5\al_1+0.5\al_2+9\seps)n
\end{align*}
vertices.
If $G'$ is non-bipartite we are done again, if not then one can apply
Lemma~\ref{f1} to find a monochromatic component with large matching in one of
the first two colors.

\smallskip

Thus, we have showed that
\begin{multline*}
R(C_{\ev {\al_1 n}},C_{\ev{\al_2 n}},C_{\od{\al_3 n}})\\
\le (\max \{2\al_1 +\al_2,\al_1+2\al_2,0.5\al_1+0.5\al_2+\al_3\}+o(1))n\,.
\end{multline*}

In order to complete the proof of Theorem~\ref{thm:main} we need to
specify colorings which result in the matching lower bound.

Again we may and shall assume that $\al_1\ge \al_2$.
Let us consider the complete graph on
$$N=2\ev{\al_1 n} +\ev{\al _2 n}-4$$
vertices whose vertex set is partitioned into four disjoint sets
$V_1$, $V_2$, $V_3$, $V_4$, where $|V_1|=|V_2|=\ev{\al_1 n}-1$ and
 $|V_3|=|V_4|=0.5\ev{\al_2 n}-1$. Let us color all edges
 contained in one of the sets $V_i$'s  with the first color,
 all edges joining  $V_1$ and $V_3$,  and those joining $V_2$ and $V_4$,
with the second color, and all other edges with the third color.
It is easy to see that this coloring gives neither a cycle
longer than $\ev{\al_1 n}-1$ in the first color,
nor a cycle longer than $\ev{\al_2 n}-2$ in the second color.
It also leads to no odd cycles in the third color. Consequently,
\begin{equation*}
R(C_{\ev {\al_1 n}},C_{\ev{\al_2 n}},C_{\od{\al_3 n}})\\
\ge 2\ev{\al_1 n} +\ev{\al _2 n}-3\,.
\end{equation*}

Finally, let us consider the complete graph on
$$\bar N=0.5\ev{\al_1 n}+0.5\ev{\al_2 n}+\od{\al_3 n}-3$$
vertices
whose set of vertices $\bar V$ is split into three parts
$\bar V_1$, $\bar V_2$, $\bar V_3$, where $|\bar V_1|=0.5\ev{\al_1 n}-1$,
 $|\bar V_2|=0.5\ev{\al_2 n}-1$, and  $|\bar V_3|=\od{\al_3 n}-1$.
Let us color edges of $K_{\bar N}$ by coloring all edges contained
in $\bar V_3$ with the third color,
all edges with at least one end in $\bar V_2$ with the second color,
and all other edges with the first color. In this coloring
there are no cycles longer then $\ev{\al_1 n}-2$ in the first color,
no cycles longer then $\ev{\al_2 n}-2$ in the second  color, and
no cycles longer then $\od{\al_3 n}-1$ colored in the third color.
Thus,
\begin{equation*}
R(C_{\ev {\al_1 n}},C_{\ev{\al_2 n}},C_{\od{\al_3 n}})\\
\ge 0.5\ev{\al_1 n}+0.5\ev{\al_2 n}+\od{\al_3 n}-2\,.
\end{equation*}
This completes the proof of Theorem~\ref{thm:main}(ii).
\end{proof}

\section{Triple of  cycles: two odd, one even}\label{sec:odd}

In this section we shall complete the proof of Theorem~\ref{thm:main},
showing the estimates for
$R(C_{\ev {\al_1 n}},C_{\od{\al_2 n}},C_{\od{\al_3 n}})$.

\begin{proof}[Proof of Theorem~\ref{thm:main}(iii)]
The proof of the upper bound for
$R(C_{\ev {\al_1 n}},C_{\od{\al_2 n}},C_{\od{\al_3 n}})$
is, again, based on Lemma~\ref{lr2}, which states that it
is enough to check that for $\al_1,\al_2,\al_3>0$,
where $c=\max \{4\al_1,\al_1+2\al_2,\al_1+2\al_3\}$,
we have $\tau_{1,2}(\al_1,\al_2,\al_3;c)$ holds.
Note that we may and shall assume that $\al_2=\al_3$, and
$\max\{\al_1,\al_2\}=1$.

Thus, for a small  $\eps >0$ consider a graph  $G=(V,E)$
obtained from the complete graph on  $N=(c+15\seps)n$ vertices by removing
at most $|E|\ge \eps^5 n^2$ edges. Let $G_1$, $G_2$, $G_3$
be a partition of $G$ induced by a three-coloring of its edges.
If the average degree $d(G_1)$ of $G_1$ is larger than $(\al_1+\eps) n$
then, by Theorem~\ref{EG}, it contains also a cycle longer than $\al_1 n +1$
and so we are done. Thus, let us assume that $d(G_1)<(\al_1+\eps)n$.
Observe that $c\ge \al_1+2\al_2$, so in this case one of the graphs
$G_2$ and $G_3$, say, $G_2$, has the average degree larger than
$$\theta n = 0.5|V|-(\al_1+\eps)n\ge (\al_2+7\seps)n\,.$$
 Suppose that $G_2$ contains no non-bipartite
component saturating at least $(\al_2 +\eps) n$ vertices. Then,
one can  apply to $G_2$ Lemma~\ref{l:nonbipartite} and decompose it
into two subgraphs $G'$ and $G''$,
where $G'$ is bipartite with bipartition $(X,Y)$, $|X|\ge |Y|$
and   $d(G'')\le \al_2 n$.  Since
$d(G'')\le \al_2 n < d(G_2)$,
we must have $d(G')\ge d(G_2)\ge \theta n$, and so
$|X|+|Y|\ge 2\theta n$, and
$|X|\ge \theta n$. Let us consider the graph $H=(V',E')$
with the vertex set $V\setminus Y$ whose edges are all edges of $G$
which are colored with the first or the third color and either
are contained in $X$, or join $X$ to $V\setminus Y$. Then $H$
is `nearly complete' graph with the hole $W=V\setminus(X\cup Y)$,
$|W|=\nu n$, colored with two colors.
 We shall show that $H$ fulfills
assumptions of Lemma~\ref{l:trzy} so that it contains either a component
in the first color containing a matching saturating at least $\al_1 n$
vertices, or a non-bipartite component in the third color with a matching
saturating at least $\al_3 n$ vertices.

Let us consider two cases.

\medskip

{\sl Case 1.} $\al_2\le 1.5\al_1$.

\smallskip

In this case we have $|V|=(4\al_1+15\seps)n$,
$\theta \ge 1.5\al_1+7\seps$,
$$\nu\le |V|/n -2\theta \le \al_1+\seps$$
and
$$|V'|\ge \frac{|V|-\nu n}{2}+\nu n=(2\al_1+0.5\nu+7\seps)n\,.$$
Thus, we have to verify that if $\al_2\le 1.5\al_1$, and $\nu\le \al_1+\seps$,
then for the function $\xi(\al_1,\al_2,\nu)$ defined by (\ref{xi}),
i.e.

\begin{multline*}
\xi(\al_1,\al_2,\nu)=\max\{\al_1+0.5\al_2, 0.5\al_1+\al_2,\\
0.5\al_1+0.5 \al_2+\nu, 2\al_1, 1.5\al_1+\nu\}\,,
\end{multline*}
 we have
\begin{equation*}
\xi(\al_1,\al_2,\nu)\le 2\al_1+0.5\nu + \seps\,.
\end{equation*}
However, the above fact follows from the definition of $\xi(\al_1,\al_2,\nu)$
and the following five inequalities:
\begin{align*}
\al_1+0.5\al_2&\leq \al_1+ 0.75\al_1\leq 2\al_1+0.5\nu\,,\\
0.5\al_1+\al_2&\le 2\al_1\le  2\al_1+0.5\nu\,,\\
0.5\al_1+0.5\al_2+\nu&=0.5\al_1+0.75\al_1+0.5\nu+0.5\nu\\
&\quad\quad\quad\quad\leq 2\al_1 +0.5\nu+\seps\,,\\
2\al_1&\leq 2\al_1+0.5\nu\,,\\
1.5\al_1+\nu&=1.5\al_1+0.5\nu+0.5\nu \leq 2\al_1+0.5\nu+\seps\,.\\
\end{align*}
Thus, the existence of a monochromatic
component which contains a large matching in either
the first or the second color follows from  Lemma~\ref{l:trzy}.

\medskip

{\sl Case 2.} $\al_2\ge 1.5\al_1$.

\smallskip

Here $|V|=(\al_1+2\al_2+15\seps)n$,
$\theta \ge \al_2+7\seps$,
$\nu \le \al_1+\seps\le \tfrac{2}{3}\al_2+\seps$,
and
$$|V'|\ge \frac{|V|-\nu n}{2}+\nu n=(0.5\al_1+\al_2+0.5\nu+7\seps)n\,.$$
Again, we check that in this case
\begin{equation}\label{eqxi}
\xi(\al_1,\al_2,\nu)\le 0.5\al_1+\al_2+0.5\nu+ \seps\,,
\end{equation}
by the direct inspection:
\begin{align*}
\al_1+0.5\al_2&\leq 0.5\al_1+ \tfrac{5}{6}\al_2\leq 0.5\al_1+\al_2+0.5\nu\,,\\
0.5\al_1+\al_2&\le 0.5\al_1+\al_2+0.5\nu\,,\\
0.5\al_1+0.5\al_2+\nu&=0.5\al_1+0.5\al_2+0.5\nu+0.5\nu\\
&\quad\quad\quad\quad\leq 0.5\al_1+\al_2+0.5\nu+\seps\,,\\
2\al_1&\leq 0.5\al_1+\al_2+0.5\nu\,,\\
1.5\al_1+\nu&\le\al_2+0.5\nu+0.5\nu \leq 0.5\al_1+\al_2+0.5\nu+\seps\,.
\end{align*}
Now we can employ Lemma~\ref{l:trzy} to complete the proof of
the lower bound for $R(C_{\ev {\al_1 n}},C_{\od{\al_2 n}},C_{\od{\al_3 n}})$.

In order to show the lower bound for
$R(C_{\ev {\al_1 n}},C_{\od{\al_2 n}},C_{\od{\al_3 n}})$
let us observe  that the same coloring we have employed for estimating
the Ramsey number for three odd cycles in the proof of
Theorem~\ref{thm:main}(iv) can be used to show that
$$R(C_{\ev{\al_1 n}},C_{\od{\al_2 n}},C_{\od{\al_3 n}})
\ge 4\od{\al_1 n}-3\,.$$

Finally, let us consider the complete graph  on
$$\tilde N= \ev{\al_1 n}+2\od{\al_2 n}-4$$
vertices whose vertex set is partitioned into four parts $\tV_1$,
$\tV_2$, $\tV_3$, $\tV_4$, such that $|\tV_1|= |\tV_2|=0.5\ev{\al_1 n}-1$,
and $|\tV_3|= |\tV_4|=\od{\al_2 n}-1$.
Let us color all  edges of $K_{\tilde N}$
contained in either $\tV_1$ or $\tV_2$,
with the first color, and use the same color
to color all edges between the pairs $\tV_1$ and $\tV_3$, and
between $\tV_2$ and $\tV_4$,
all edges contained in either $\tV_3$ or $\tV_4$ we color
with the second color,  and all other edges with the
third color. It can be easily seen that in this coloring no cycle
longer than $0.5\ev{\al_1 n}-2$ is colored with the first color,
no cycle longer than $\od{\al_2 n}-1$ is colored with the second color,
and no odd cycle is colored with the third color.
Consequently,
$$R(C_{\ev{\al_1 n}},C_{\od{\al_2 n}},C_{\od{\al_3 n}})
\ge \ev{\al_1 n}+2\od{\al_2 n}-3\,.$$
An analogous construction gives
$$R(C_{\ev{\al_1 n}},C_{\od{\al_2 n}},C_{\od{\al_3 n}})
\ge \ev{\al_1 n}+2\od{\al_3 n}-3\,.$$
This completes the proof of (iii) and  Theorem~\ref{thm:main}.
\end{proof}

\bibliographystyle{amsplain}

\end{document}